\documentclass[11pt]{article}
\usepackage{amssymb,amsmath,amsfonts,amsthm}

\usepackage{epsfig}

\numberwithin{equation}{section}

\newcommand{\R}{{\mathbb R}}

\newcommand{\intr}{\int\limits_{\R^N}}

\theoremstyle{plain}

\newtheorem{proposition}{Proposition}[section]

\newtheorem{theorem}{Theorem}[section]
\newtheorem{lemma}{Lemma}[section]

\theoremstyle{definition}

\begin{document}

\title{\vskip-0.3in A study of biharmonic equation involving nonlocal terms and critical Sobolev exponent.}

\author{Gurpreet Singh\footnote{School of Mechanical and Manufacturing Engineering, Dublin City University; {\tt gurpreet.bajwa2506@gmail.com}}}

\maketitle

\begin{abstract}
In this paper, we investigate the existence of ground state solutions and non-existence of non-trivial weak solution of 
\begin{equation*}
\Delta^{2} u= \Big(|x|^{-\theta}*|u|^{p_{\theta}}\Big)|u|^{p_{\theta}-2}u+\alpha\Big(|x|^{-\gamma}*|u|^{p}\Big)|u|^{p-2}u \quad\mbox{ in }\R^{N},
\end{equation*}
where $0<p\leq p_{\gamma}^{*}$, $\alpha>0$, $\theta, \gamma \in (0,N)$, $p_{\theta}=\frac{2(N-\theta)}{N-4}$ and $N\geq 5$. Firstly, we prove the non-existence by establishing Pohozaev type of identity. Next, we study the existence of ground state solutions by using the minimization method on the associated Nehari manifold.
\end{abstract}




\section{Introduction}\label{sec1}
We study the existence of ground state solutions and non-existence of non-trivial weak solution for the following problem
\begin{equation}\label{fr}
\Delta^{2} u= \Big(|x|^{-\theta}*|u|^{p_{\theta}}\Big)|u|^{p_{\theta}-2}u+\alpha\Big(|x|^{-\gamma}*|u|^{p}\Big)|u|^{p-2}u \quad\mbox{ in }\R^{N},
\end{equation}
where $\alpha> 0$, $\theta, \gamma \in (0,N)$, $p_{\theta}=\frac{2(N-\theta)}{N-4}$ and $N\geq 5$. Here, $p$ satisfies $0<p\leq p_{\gamma}^{*}$and $p_{\gamma}^{*}= \frac{2N-\gamma}{N-4}$ is the critical Sobolev exponent. $\Delta^{2}= \Delta(\Delta)$ is the biharmonic operator and $|x|^{-\xi} $ is the Riesz potential of order $\xi \in (0, N)$. 

The case when $\alpha=0$, \eqref{fr} becomes the biharmonic Choquard equation
\begin{equation}\label{squ}
\Delta^{2} u= \Big(|x|^{-\theta}*|u|^{p_{\theta}}\Big)|u|^{p_{\theta}-2}u \quad\mbox{ in }\R^{N}.
\end{equation}

In the last few decades, the Choquard equation has received a great attention and has been appeared in many different contexts and settings(see \cite{AFY2016, AGSY2017, GT2016, MV2017, MS2017}). The following Choquard or nonlinear Schr\"odinger-Newton equation
\begin{equation}\label{ce}
-\Delta u+ u= (|x|^{-2}*u^2)u \quad\mbox{ in }\R^{N},
\end{equation}
was first considered by Pekar\cite{P1954} in 1954 for $N=3$. In 1996, Penrose had used the equation \eqref{ce} as a model in self-gravitating matter(see \cite{P1996}, \cite{P1998}). The stationary Choquard equation
$$
-\Delta u+ V(x)u= (|x|^{-\theta}*|u|^b)|u|^{b-2}u \quad\mbox{ in }\R^{N},
$$
arises in quantum theory and in the theory of Bose-Einstein condensation. 

Equations involving biharmonic operator arise in many real life phenomena such as in biophysics, continuum mechanics, differential geometry and many more. For example, in the modeling of thin elastic plates, clamped plates and in the study of Paneitz-Branson equation and the Willmore equation(see \cite{GGS2010}). As we could not apply maximum principle for the biharmonic operator which makes the problems involving biharmonic operator even more interesting from the mathematical point of view(see \cite{BG2011, JC2014, PS2012, PS2014, YT2013, ZTZ2014})

In the recent years, the biharmonic equations has received a considerable attention. In \cite{CD2018}, Cao and Dai studied the following biharmonic equation with Hatree type nonlinearity
\begin{equation*}
\Delta^{2} u= \Big(|x|^{-8}*|u|^2\Big)|u|^{b}, \quad\mbox{ for all } x\in \R^d,
\end{equation*}
where $0< b\leq 1$ and $d\geq 9$. The authors applied the methods of moving plane and proved that the non-negative classical solutions are radially symmetric. The authors also studied the non-existence of non-trivial non-negative classical solutions in subcritical case $0<b<1$.

Micheletti and Pistoiain \cite{MP1998} has investigated the following problem
\begin{equation*}
\left\{
\begin{aligned}
\Delta^{2} u+c\Delta u&= f(x, u) &&\quad\mbox{ in } \Omega,\\
u&= \Delta u= 0 &&\quad\mbox{ on } \partial{\Omega},
\end{aligned}
\right.
\end{equation*}
where $\Omega$ is smooth bounded domain in $\R^N$. The authors obtained the multiple non-trivial solutions by using the mountain pass theorem. Existence of infinitely many sign-changing solutions of the above problem had been studied by Zhao and Xu in \cite{ZW2008} by the use of critical point theorem.

\smallskip

In this research article, we investigate the existence of ground state solutions and non-existence of non-trivial weak solution to \eqref{fr}. Now, in the subsection below we provide some notations which we will using throught this paper, variational framework and main results. 

\smallskip

\subsection{Notations and Variational Framework}

\subsection*{Notations.} In this paper, we will be using the following notations.
\begin{itemize}
\item $H_{0}^{2}(\R^N)= W_{0}^{2,2}(\R^N)$ is the Hilbert-Sobolev space endowed with the inner product
$$
<u, v>_{H_{0}^{2}(\R^N)}= \intr \Delta u \Delta v dx,
$$
and norm
$$
||u||_{H_{0}^{2}(\R^N)}= \Big( \intr |\Delta u|^2 dx \Big)^{\frac{1}{2}}.
$$
\item $L^{t}(\R^N)$ denotes the usual Lebesgue space in $\R^N$ of order $s\in [1 , \infty]$ whose norm will be denoted by $||.||_{t}$.
\item $\hookrightarrow$ denotes the continuous embeddings.
\end{itemize}

\medskip

{\bf Note: }The embedding
$$
H_{0}^{2}(\R^N) \hookrightarrow L^{t}(\R^N),
$$
is continuous for all $1\leq t\leq \frac{2N}{N-4}$.

\subsection*{Variational Framework.}
\begin{itemize}
\item We will require the following Hardy-Littlewood-Sobolev inequality
\begin{equation}\label{hli}
\Big| \int_{\R^N} \Big(|x|^{-\xi}*u\Big)v \Big| \leq C\|u\|_r \|v\|_s,	
\end{equation}
for $\delta \in (0, N)$, $\mu \geq 0$, $u\in L^{r}(\R^N)$ and $v\in L^{s}(\R^N)$ such that
$$
\frac{1}{r}+ \frac{1}{s}+ \frac{\xi}{N}= 2.
$$

\item We assume that $p_{\theta}$ satisfies 
\begin{equation}\label{p}
\frac{2N-\theta}{2N}< p_{\theta}< \frac{2N-\theta}{N-4},
\end{equation}
and $p$ satisfies
\begin{equation}\label{p1}
\frac{2N-\gamma}{2N}< p\leq p_{\gamma}^{*}.
\end{equation}

\item We could easily notice that the equation \eqref{fr} has variational structure. Now, let us define the energy functional ${\mathcal F}_\alpha:H_{0}^{2}(\R^{N}) \rightarrow \R$ by
\begin{equation}\label{fr1}
\begin{aligned}
{\mathcal F}_\alpha(u)&=\frac{1}{2}\|u\|_{H_{0}^{2}(\R^{N})}^{2}-\frac{1}{2p_{\theta}}\int_{\R^{N}}\Big(|x|^{-\theta}*|u|^{p_{\theta}}\Big)|u|^{p_{\theta}}-\frac{\alpha}{2p}\int_{\R^{N}}\Big(|x|^{-\gamma}*|u|^{p}\Big)|u|^{p}.
\end{aligned}
\end{equation}
\end{itemize}
By using \eqref{p} and \eqref{p1} together with the Hardy-Littlewood-Sobolev inequality \eqref{hli} we get that the energy functional ${\mathcal F}_\alpha$ is well defined  and ${\mathcal F}_\alpha\in C^1(H_{0}^{2}(\R^{N}))$. Also, solution of \eqref{fr} is a critical point of the energy functional  ${\mathcal F}_\alpha$.

\subsection{Main Results.} 

\subsection*{Non-existence ($p< p_{\gamma}^{*})$}
Firstly, we study the non-existence of non-trivial weak solution to \eqref{fr}. By weak solution of \eqref{fr} one could understand that there exists $u\in H_{0}^{2}(\R^N)$, $u\neq 0$ and 
$$
\intr \Delta u \Delta v dx= \int_{\R^{N}}\Big(|x|^{-\theta}*|u|^{p_{\theta}-2}\Big)|u|^{p_{\theta}-2}v dx-\int_{\R^{N}}\Big(|x|^{-\gamma}*|u|^{p-2}\Big)|u|^{p-2}v dx,
$$
for all $v\in C_{c}^{\infty}(\R^N)$. 

Now, we present our main result on non-existence of non-trivial weak solution.

\begin{theorem}\label{nonexis}
Let $N\geq 5$, $\alpha> 0$, $\frac{2N-\gamma}{2N}< p< p_{\gamma}^{*}$ and $p_{\theta}$ satisfies \eqref{p}. If $u\in H_{0}^{2}(\R^N)$ is a weak solution of \eqref{fr}, then $u\equiv 0$.
\end{theorem}

\smallskip

Next, we investigate the existence of ground state solutions for the equation \eqref{fr}.

\subsection*{Existence $(p=p_{\gamma}^{*})$}
Define the Nehari manifold associated with the energy functional $F_{\alpha}$ by 
\begin{equation}\label{nm}
{\mathcal N_\alpha}=\{u\in H_{0}^{2}(\R^N)\setminus\{0\}: \langle {\mathcal F}_\alpha'(u),u\rangle=0\},
\end{equation}
and the ground state solutions will be obtained as minimizers of
$$
b_{\alpha}=\inf_{u\in {\mathcal N_\alpha}}{\mathcal F}_\alpha(u).
$$
We present the main result on ground state solutions.
\begin{theorem}\label{gstate}
Assume that $N\geq 5$, $2\theta< \gamma$, $p_{\theta}> p_{\gamma}^{*}> 1$ and $\alpha> 0$. Further, if $p_{\theta}$ satisfies \eqref{p}, then the equation \eqref{fr} has a ground state solution $u\in H_{0}^{2}(\R^{N})$. 
\end{theorem}

\medskip

Now, in the Section $2$	, we will be collecting some preliminary results and then it will be followed by Section $3$ and $4$ which consists of the proofs of our main results.

\bigskip

\section{Preliminary results}

\begin{lemma}\label{cc}(\cite[Lemma 1.1]{L1984}, \cite[Lemma 2.3]{MV2013})
There exists a constant $C_0>0$ such that
$$
\int_{\R^N}|u|^t \leq C_0||u||\Big(\sup_{y\in \R^N} \int_{B_1(y)}|u|^t \Big)^{1-\frac{2}{t}},
$$
for any $u\in H_{0}^{2}(\R^N)$ and $t\in [1, \frac{2N}{N-4}]$.
\end{lemma}

\begin{lemma}\label{bogachev}(\cite[Proposition 4.7.12]{B2007})
Let $1< t< \infty$ and assume that $(z_n)$ is a bounded sequence in $L^t(\R^N)$ which converges to $z$ almost everywhere. Then, $z_n$ converges weakly to $z$ in  $L^t(\R^N)$, that is, $z_n\rightharpoonup z$.
\end{lemma}

\begin{lemma}\label{blc}{\sf (\it Local Brezis-Lieb lemma)}
Let $1< t< \infty$ and assume that $(z_n)$ is a bounded sequence in $L^t(\R^N)$ which converges to $z$ almost everywhere. Then, for every $1\leq q\leq t$, we have
$$
\lim_{n\to\infty}\int_{\R^N}\big| |z_n|^q-|z_n-z|^q-|z|^q\big|^{\frac{t}{q}}=0\,,
$$
and
$$
\lim_{n\to\infty}\int_{\R^N}\big| |z_n|^{q-1}z_n-|z_n-z|^{q-1}(z_n-z)-|z|^{q-1}z\big|^{\frac{t}{q}}=0.
$$
\end{lemma}
\begin{proof}
Let us fix $\varepsilon>0$, then there exists a constant $C(\varepsilon)>0$ such that for all $c$,$d\in \R$, one could have
\begin{equation}\label{lb1}
\Big||c+d|^{q}-|c|^{q}\big|^{\frac{t}{q}}\leq \varepsilon|c|^{t}+C(\varepsilon)|d|^{t},
\end{equation}
Next, using \eqref{lb1}, we get
$$
\begin{aligned}
|f_{n, \varepsilon}|=& \Big( \Big||z_{n}|^{q}-|z_{n}-z|^{q}-|z^{q}|\Big|^{\frac{t}{q}}-\varepsilon|z_{n}-z|^{t}\Big)^{+}\\
&\leq (1+C(\varepsilon))|z|^{t}.
\end{aligned}
$$
Further, using the Lebesgue Dominated Convergence theorem, we obtain
\begin{equation}
\intr {f_{n, \varepsilon}} \rightarrow 0  \quad\mbox{ as } n\rightarrow \infty.
\end{equation}
Thus, we find that
$$
\Big||z_{n}|^{q}-|z_{n}-z|^{q}-|z|^{q}\Big|^{\frac{t}{q}}\leq f_{n, \varepsilon}+\varepsilon|z_{n}-z|^{t},
$$
which yields 
$$
\limsup_{n\rightarrow \infty} {\intr \Big||z_{n}|^{q}-|z_{n}-z|^{q}-|z|^{q}\Big|^{\frac{t}{q}}}\leq c\varepsilon,
$$
for some $c= \sup_{n}|z_{n}-z|_{t}^{t}< \infty$. We finish the proof by letting $\varepsilon \rightarrow 0$.
\end{proof}

\begin{lemma}\label{nlocbl}{\sf (\it Nonlocal Brezis-Lieb lemma}(\cite[Lemma 2.4]{MV2013})
Let us assume that $N\geq 5$, $\theta\in (0,N)$ and $p_{\theta}\in [1,\frac{2N}{2N-\theta})$. Suppose that $(u_n)$ is a bounded sequence in $L^{\frac{2Np_{\theta}}{2N-\theta}}(\R^N) \cap L^{\frac{2N}{N-4}}(\R^N) $ such that $u_n \rightarrow u$ almost everywhere in $\R^N$. Then, we have
\begin{equation*}
\int_{\R^N}\Big(|x|^{-\theta}*{|u_n|^{c}}\Big){|u_n|^c}dx-\int_{\R^N}\Big(|x|^{-\theta}*{|u_n-u|^c}\Big){|u_n-u|^c}dx \rightarrow \int_{\R^N}\Big(|x|^{-\theta}*{|u|^c}\Big){|u|^c}dx,
\end{equation*}
where either $c= p_{\theta}$ or $c=p_{\gamma}^{*}=\frac{2N-\gamma}{N-4}$\; {\rm (In this case $\theta=\gamma$)}.
\end{lemma}
\begin{proof}
Let us take $q=c=p_{\theta}$, $t=\frac{2N}{2N-\theta}$ in Lemma \ref{blc}, then we get
\begin{equation}\label{bh1}
|u_n-u|^{p_{\theta}}-|u_n|^{p_{\theta}}\to |u|^{p_{\theta}} \mbox{ strongly in } L^{\frac{2N}{2N-\theta}}(\R^N),
\end{equation}
as $n\rightarrow \infty$. By Lemma \ref{bogachev}, we have 
\begin{equation}\label{bh2}
|u_n-u|^{p_{\theta}}\rightharpoonup 0 \mbox{ weakly in } L^{\frac{2N}{2N-\theta}}(\R^{N}).
\end{equation}
Next, we use the Hardy-Littlewood-Sobolev inequality \eqref{hli} to obtain
\begin{equation}\label{b3}
|x|^{-\theta}*\Big(|u_n-u|^p_{\theta}-|u_n|^{p_{\theta}}\Big)\to |x|^{-\theta}*|u|^{p_{\theta}} \quad\mbox{ in } L^{\frac{2N}{\theta}}(\R^N).
\end{equation}
On the other hand
\begin{equation}\label{nb1}
\begin{aligned}
&\int_{\R^N}\Big(|x|^{-\theta}*|u_n|^{p_{\theta}}\Big)|u_n|^{p_{\theta}}dx-\int_{\R^N}\Big(|x|^{-\theta}*|u_n-u|^{p_{\theta}}\Big)|u_n-u|^{p_{\theta}} dx\\&=\intr \Big[|x|^{-\theta}*\Big(|u_n|^{p_{\theta}}-|u_n-u|^{p_{\theta}}\Big)\Big]\Big(|u_n|^{p_{\theta}}-|u_n-u|^{p_{\theta}}\Big)dx\\
&+2\intr \Big[|x|^{-\theta}*\Big(|u_n|^{p_{\theta}}-|u_n-u|^{p_{\theta}}\Big)\Big]|u_n-u|^{p_{\theta}} dx.
\end{aligned}
\end{equation}

Finally, passing to the limit in \eqref{nb1} and using \eqref{bh1}-\eqref{bh2}, the result holds. Similarly, we prove the case when $c=p_{\gamma}^{*}$.
	
\end{proof}

\begin{lemma}\label{anbl}
Let us assume that $N\geq 5$, $\theta\in (0,N)$ and $p_{\theta}\in [1,\frac{2N}{2N-\theta})$. Suppose that $(u_n)$ is a bounded sequence in $L^{\frac{2Np_{\theta}}{2N-\theta}}(\R^N) \cap L^{\frac{2N}{N-4}}(\R^N) $ such that $u_n \rightarrow u$ almost everywhere in $\R^N$. Then, for any $h\in L^{\frac{2Np_{\theta}}{2N-\theta}}(\R^N)\cap L^{\frac{2N}{N-4}}(\R^N)$ we have
$$
\int_{\R^N}\Big(|x|^{-\theta}*|u_n|^c\Big)|u_n|^{c-2}u_nh\; dx \rightarrow \int_{\R^N} \Big(|x|^{-\theta}*|u|^c\Big)|u|^{c-2}uh\; dx,
$$
where either $c= p_{\theta}$ or $c=p_{\gamma}^{*}=\frac{2N-\gamma}{N-4}$ \;{\rm (In this case $\theta=\gamma$)}.
\end{lemma}
\begin{proof} 
We will prove the lemma for $c=p_{\theta}$ and the method will be similar to prove it for the the second case, that is, $c= p_{\gamma}^{*}$. Let us assume that $h=h^+-h^-$ and $v_n=u_n-u$. We only require to prove the lemma for $h\geq 0$. We use Lemma \ref{blc} with $q=c=p_{\theta}$ and $t=\frac{2N}{2N-\theta}$ together with $(z_n,z)=(u_n,u)$ and $(z_n,z)=(u_nh^{1/c}, u h^{1/c})$ respectively in order to obtain
$$
\left\{
\begin{aligned}
&|u_n|^{p_{\theta}}-|v_n|^{p_{\theta}}\to |u|^{p_{\theta}} \\
&|u_n|^{{p_{\theta}}-2}u_n h- |v_n|^{{p_{\theta}}-2}v_n h\to |u|^{{p_{\theta}}-2}uh
\end{aligned}
\right.
\quad\mbox{ strongly in }\; L^{\frac{2N}{2N-\theta}}(\R^N).
$$
Now, we use the Hardy-Littlewood-Sobolev inequality to get
\begin{equation}\label{est00}
\left\{
\begin{aligned}
&|x|^{-\theta}*\Big(|u_n|^{p_{\theta}}-|v_n|^{p_{\theta}}\Big)\to |x|^{-\theta}*|u|^{p_{\theta}} \\
&|x|^{-\theta}*\Big(|u_n|^{{p_{\theta}}-2}u_nh-|v_n|^{{p_{\theta}}-2}v_nh\Big)\to |x|^{-\theta}*\Big(|u|^{{p_{\theta}}-2}uh\Big)
\end{aligned}
\right.
\quad\mbox{ strongly in }\; L^{\frac{2N}{\theta}}(\R^N).
\end{equation}
Also, Lemma \ref{bogachev} yields
\begin{equation}\label{est01}
\left\{
\begin{aligned}
&|u_n|^{{p_{\theta}}-2}u_n h\rightharpoonup |u|^{{p_{\theta}}-2}uh\\ 
& |v_n|^{p_{\theta}}\rightharpoonup 0\\
&|v_n|^{{p_{\theta}}-2}v_nh\rightharpoonup 0
\end{aligned}
\right.
\quad \mbox{ weakly in }\; L^{\frac{2N}{2N-\theta}}(\R^N)
\end{equation}
Then, the combination of \eqref{est00} and \eqref{est01} will give us
\begin{equation}\label{est02}
\begin{aligned}
&\intr \Big[|x|^{-\theta}*\Big(|u_n|^{p_{\theta}}-|v_n|^{p_{\theta}}\Big)\Big]\Big(|u_n|^{{p_{\theta}}-2}u_nh-|v_n|^{{p_{\theta}}-2}v_n h\Big)\to \int_{\R^N} \Big(|x|^{-\theta}*|u|^{p_{\theta}}\Big)|u|^{{p_{\theta}}-2}uh,\\
& \intr \Big[|x|^{-\theta}*\Big(|u_n|^{p_{\theta}}-|v_n|^{p_{\theta}}\Big)\Big]|v_n|^{{p_{\theta}}-2}v_nh\to 0,\\
&\intr \Big[|x|^{-\theta}*\Big(|u_n|^{{p_{\theta}}-2}u_nh-|v_n|^{{p_{\theta}}-2}v_nh\Big)\Big]|v_n|^{p_{\theta}}\to 0.
\end{aligned}
\end{equation}
Next, we use the Hardy-Littlewood-Sobolev inequality together with H\"older's inequality in order to find
\begin{equation}\label{est03}
\begin{aligned}
\left| \intr \Big(|x|^{-\theta}*|v_n|^{{p_{\theta}}}\Big)|v_n|^{{p_{\theta}}-2}v_nh \right|& \leq \|v_n\|^{p_{\theta}}_{\frac{2N{p_{\theta}}}{2N-\theta}}\||v_n|^{{p_{\theta}}-1}h\|_{\frac{2N}{2N-\theta}}\\
&\leq {p_{\theta}} \||v_n|^{{p_{\theta}}-1}h\|_{\frac{2N}{2N-\theta}}.
\end{aligned}
\end{equation}
Furthermore, by Lemma \ref{bogachev} we have $v_n^{\frac{2N(p_{\theta}-1)}{2N-\theta}}\rightharpoonup 0$ weakly in $L^{\frac{p_{\theta}}{p_{\theta}-1}}(\R^N)$. Hence, 
$$
\||v_n|^{p_{\theta}-1}h\|_{\frac{2N}{2N-\theta}}=\left(\intr |v_n|^{\frac{2N(p_{\theta}-1)}{2N-\theta}}|h|^{\frac{2N}{2N-\theta}}  \right)^{\frac{2N-\theta}{2N}}\to 0.
$$
Therefore, using \eqref{est03} we obtain
\begin{equation}\label{est04}
\lim_{n\to \infty} \intr \Big(|x|^{-\theta}*|v_n|^{p_\theta}\Big)|v_n|^{p_{\theta}-2}v_nh=0.
\end{equation}
On the other hand, we notice that
\begin{equation}\label{bl1}
\begin{aligned}
\intr \Big(|x|^{-\theta}*|u_n|^{p_{\theta}}\Big)|u_n|^{{p_{\theta}}-2}u_nh=& \intr \Big[|x|^{-\theta}*\Big(|u_n|^{p_{\theta}}-|v_n|^{p_{\theta}}\Big)\Big]\Big(|u_n|^{{p_{\theta}}-2}u_nh-|v_n|^{{p_{\theta}}-2}v_nh\Big)\\
&+\intr \Big[|x|^{-\theta}*\Big(|u_n|^{p_{\theta}}-|v_n|^{p_{\theta}}\Big)\Big]|v_n|^{{p_{\theta}}-2}v_nh\\
&+\intr \Big[|x|^{-\theta}*\Big(|u_n|^{{p_{\theta}}-2}u_nh-|v_n|^{{p_{\theta}}-2}v_n h\Big)\Big]|v_n|^{p_{\theta}}\\
&+\intr \Big(|x|^{-\theta}*|v_n|^{p_{\theta}}\Big)|v_n|^{{p_{\theta}}-2}v_nh.
\end{aligned}
\end{equation}
Finally, by passing to the limit in \eqref{bl1} together with \eqref{est02} and \eqref{est04} we get the desired result.
\end{proof}

\bigskip

\section{Proof of Theorem \ref{nonexis}}

In order to prove this theorem, we establish the following Pohozaev type of identity.

\begin{proposition}
Assume that $u\in H_{0}^{2}(\R^N)$ is a solution of \eqref{fr}. Then, we have
\begin{equation}\label{ne5}
\Big(\frac{N-4}{2}\Big) \|u\|_{H_{0}^2(\R^N)}^{2}= \frac{2N-\theta}{2p_{\theta}}\intr \Big(|x|^{-\theta}*|u|^{p_{\theta}}\Big)|u|^{p_{\theta}}+\alpha \Big(\frac{2N-\gamma}{2p}\Big)\intr \Big(|x|^{-\gamma}*|u|^{p}\Big)|u|^{p}.
\end{equation}
\end{proposition}
\begin{proof}
Let us define the cut-off function by
$$
\varphi_{\kappa, \delta}(x)= \phi_{\kappa}(x)\psi_{\delta}(x),
$$
where 
\begin{equation*}
\phi_{\kappa}(x)= \phi\Big(\frac{|x|}{\kappa}\Big) \quad\mbox{ and } \psi_{\delta}(x)= \psi \Big(\frac{|x|}{\delta}\Big),
\end{equation*}
are smooth functions. The functions $ \phi_{\kappa}(x)$ and $\psi_{\delta}(x)$ satisfy the following properties:
\begin{itemize}
\item $0\leq \phi_{\kappa}(x), \psi_{\delta}(x) \leq 1$.
\item ${\rm supp}  \phi= (1, \infty)$ and $ {\rm supp} \psi= (-\infty, 2)$.
\item $\phi(x) =1$ for all $x\geq 2$ and $\psi(x)= 1$ for all $x\leq 1$.
\end{itemize}
Also, $u$ is a smooth function away from the origin (see \cite{GGS2010}) and $(x. \nabla u)\varphi_{\kappa, \delta} \in C_{c}^{3}(\R^N)$. Next, we multiply the equation \eqref{fr} by $(x. \nabla u)\varphi_{\kappa, \delta}$ and we obtain

\begin{equation}\label{ne1}
\intr \Delta^{2} u (x. \nabla u)\varphi_{\kappa, \delta} dx= \intr \Big(|x|^{-\theta}*|u|^{p_{\theta}}\Big)|u|^{p_{\theta}-2}u (x. \nabla u)\varphi_{\kappa, \delta} dx + \alpha \intr \Big(|x|^{-\gamma}*|u|^{p}\Big)|u|^{p-2}u (x. \nabla u)\varphi_{\kappa, \delta} dx.
\end{equation}
Following the similar approach as in \cite{Barxiv}, we get
\begin{equation}\label{ne2}
\lim_{\delta \to \infty}\lim_{\kappa \to 0} \intr \Delta^{2} u (x. \nabla u)\varphi_{\kappa, \delta} dx = -\Big(\frac{N-4}{2}\Big)\intr |\Delta u|^2.
\end{equation}
On the other hand, we have
$$
\begin{aligned}
\intr \Big(|x|^{-\theta}*|u|^{p_{\theta}}\Big)|u|^{p_{\theta}-2}u (x. \nabla u)\varphi_{\kappa, \delta} dx \\
=&\frac{1}{p_{\theta}}\intr\intr |x-y|^{-\theta}|u|^{p_{\theta}}(y) \varphi_{\kappa, \delta}(x) x\cdot \nabla(|u|^{p_{\theta}})(x)dxdy\\
=&\frac{1}{2 p_{\theta}}\intr\intr |x-y|^{-\theta}\Big\{ |u|^{p_{\theta}}(y) \varphi_{\kappa, \delta}(x) x\cdot \nabla(|u|^{p_{\theta}})(x)\\
&+
|u|^{p_{\theta}}(x) \varphi_{\kappa, \delta}(y) y\cdot \nabla(|u|^{p_{\theta}})(y)\Big\}dxdy\\
=&-\frac{1}{p_{\theta}}\intr\intr |x-y|^{-\theta}|u|^{p_{\theta}}(x)|u|^{p_{\theta}}(y)\Big[N\varphi_{\kappa, \delta}(x)+ x\cdot \nabla \varphi_{\kappa, \delta}(x)\Big] dxdy\\
&+\frac{\theta}{2 p_{\theta}}\intr\intr \frac{(x-y)(x\varphi_{\kappa, \delta}(x)-y\varphi_{\kappa, \delta}(y))}{|x-y|^2}|x-y|^{-\theta} |u|^{p_{\theta}}(x)|u|^{p_{\theta}}(y)dx dy,
\end{aligned}
$$
which further yields
\begin{equation}\label{ne3}
\lim_{\delta\to \infty} \lim_{\kappa \to 0} \intr \Big(|x|^{-\theta}*|u|^{p_{\theta}}\Big)|u|^{p_{\theta}-2}u (x. \nabla u)\varphi_{\kappa, \delta} dx =-\frac{2N-\theta}{2p_{\theta}}\intr \Big(|x|^{-\theta}*|u|^{p_{\theta}}\Big)|u|^{p_{\theta}}.
\end{equation}
Similarly, we get
\begin{equation}\label{ne4}
\lim_{\delta\to \infty} \lim_{\kappa \to 0} \intr \Big(|x|^{-\gamma}*|u|^{p}\Big)|u|^{p-2}u (x. \nabla u)\varphi_{\kappa, \delta} dx =-\frac{2N-\gamma}{2p}\intr \Big(|x|^{-\gamma}*|u|^{p}\Big)|u|^{p}.
\end{equation}
Next, passing to limit in \eqref{ne1} and using \eqref{ne2}, \eqref{ne3} and \eqref{ne4}, we get the result.

\end{proof}

\smallskip

\subsection{Completition of Proof of Theorem \ref{nonexis}}
Since, $u$ is a solution of \eqref{fr}, we also have 
\begin{equation}\label{ne6}
\|u\|_{H_{0}^2(\R^N)}^{2}= \intr \Big(|x|^{-\theta}*|u|^{p_{\theta}}\Big)|u|^{p_{\theta}}+\alpha \intr \Big(|x|^{-\gamma}*|u|^{p}\Big)|u|^{p}.
\end{equation}
Using \eqref{ne6} in \eqref{ne5}, we get
\begin{equation*}
0= \Big(\frac{N-4}{2}-\frac{2N-\theta}{2p_{\theta}}\Big)\intr \Big(|x|^{-\theta}*|u|^{p_{\theta}}\Big)|u|^{p_{\theta}}+\alpha \Big(\frac{N-4}{2}-\frac{2N-\gamma}{2p}\Big)\intr \Big(|x|^{-\gamma}*|u|^{p}\Big)|u|^{p}.
\end{equation*}
Now, we use the Hardy-Littlewood-Sobolev inequality together with the fact that the embeddings  $H_{0}^{2}(\R^N) \hookrightarrow L^{\frac{2Np_{\theta}}{2N-\theta}}(\R^N)$ and  $H_{0}^{2}(\R^N) \hookrightarrow L^{\frac{2Np}{2N-\gamma}}(\R^N)$ are continuous in order to obtain
$$
\Big(\frac{2N-\theta-p_{\theta}(N-4)}{2p_{\theta}}\Big)\|u\|_{H_{0}^2(\R^N)}^{2(p_{\theta}-p)}\leq \alpha \Big(\frac{N-4}{2}-\frac{2N-\gamma}{2p}\Big).
$$
The above inequality is imposssible as $\Big(\frac{2N-\theta-p_{\theta}(N-4)}{2p_{\theta}}\Big)\|u\|_{H_{0}^2(\R^N)}^{2(p_{\theta}-p)}> 0$ and on the other hand $p<p_{\gamma}^{*}$. This concludes our proof.

\bigskip

Now, we look into the ground state solutions to \eqref{fr} in the following section.

\medskip

\section{Proof of Theorem \ref{gstate}}
Our proof will be relying on the analysis of the Palais-Smale sequences for ${\mathcal F}_\alpha \!\mid_{\mathcal N_\alpha}$. We will follow the ideas from \cite{CM2016,CV2010} in order to prove that any Palais-Smale sequence of ${\mathcal F}_\alpha \!\mid_{\mathcal N_\alpha}$ is either converging strongly to its weak limit or differs from it by a finite number of sequences, which are the translated solutions of \eqref{squ}. We will be using several nonlocal Brezis-Lieb results which we have presented in Section $2$.
Assume $\alpha> 0$.  For $u,v\in H_{0}^{2}(\R^{N})$  we have
\begin{equation*}
\begin{aligned}
\langle {\mathcal F}'_\alpha (u), v \rangle&= \int_{\R^{N}} \Delta u \Delta v  -\int_{\R^{N}}\Big(|x|^{-\theta}*|u|^{p_{\theta}}\Big)|u|^{p_{\theta}-1}v -\alpha\int_{\R^{N}}\Big(|x|^{-\gamma}*|u|^{p_{\gamma}^{*}}\Big)|u|^{p_{\gamma}^{*}-1}v.
\end{aligned}
\end{equation*}
Also, we have
\begin{equation*}
\begin{aligned}
\langle {\mathcal F}'_\alpha(su), su \rangle &= s^{2}\|u\|_{H_{0}^2(\R^N)}^{2}- s^{2 p_{\theta}}\int_{\R^{N}}\Big(|x|^{-\theta}*|u|^{p_{\theta}}\Big)|u|^{p_{\theta}}-\alpha s^{2p_{\gamma}^{*}}\int_{\R^{N}}\Big(|x|^{-\gamma}*|u|^{p_{\gamma}^{*}}\Big)|u|^{p_{\gamma}^{*}},
\end{aligned}
\end{equation*}
for some $s>0$.

Since $p_{\theta}>p_{\gamma}^{*}> 1$, this gives us that the equation $\langle {\mathcal F}_\alpha'(su),su \rangle= 0 $ has a unique positive solution $s=s(u)$, also known as the {\it projection of $u$} on ${\mathcal N_\alpha}$. Now, we will discuss the main properties of the Nehari manifold ${\mathcal N_\alpha}$:

\begin{lemma}\label{nehari1}
\begin{itemize}
\item [(i)]${\mathcal F}_\alpha \geq c||u||_{H_{0}^{2}(\R^N)}$ for some constant $c>0$, that is, ${\mathcal F}_\alpha$ is coercive. 
\item [(ii)] ${\mathcal F}_\alpha \!\mid_{\mathcal N_\alpha}$ is bounded from below by a positive constant.
\end{itemize}
\end{lemma}
\begin{proof} 
\begin{itemize}
\item [(i)] One could notice that
$$
\begin{aligned}
{\mathcal F}_\alpha(u)&= {\mathcal F}_\alpha(u)-\frac{1}{2p_{\gamma}^{*}}\langle {\mathcal F}_\alpha'(u), u \rangle \\
&=\Big(\frac{1}{2}-\frac{1}{2p_{\gamma}^{*}}\Big)\|u\|_{H_{0}^{2}(\R^N)}^2+\Big(\frac{1}{2p_{\gamma}^{*}}-\frac{1}{2.p_{\theta}}\Big)\int_{\R^{N}} \Big(|x|^{-\theta}*|u|^{p_{\theta}}\Big)|u|^{p_{\theta}} \\
&\geq \Big(\frac{1}{2}-\frac{1}{2p_{\gamma}^{*}}\Big)\|u\|_{H_{0}^{2}(\R^N)}^2.
\end{aligned}
$$
We conclude the proof by taking $c=\frac{1}{2}-\frac{1}{2p_{\gamma}^{*}}$.

\smallskip

\item [(ii)] We will be using the Hardy-Littlewood-Sobolev inequality and the fact that the embeddings  $H_{0}^{2}(\R^N) \hookrightarrow L^{\frac{2Np_{\theta}}{2N-\theta}}(\R^N)$ and  $H_{0}^{2}(\R^N) \hookrightarrow L^{\frac{2N}{N-4}}(\R^N)$ are continuous. Let $u\in {\mathcal N_\alpha}$, then we have
$$
\begin{aligned}
0=\langle {\mathcal F}_\alpha'(u),u\rangle& =\|u\|_{H_{0}^{2}(\R^N)}^2-\int_{\R^{N}}\Big(|x|^{-\theta}*|u|^{p_{\theta}}\Big)|u|^{p_{\theta}}-\alpha \int_{\R^{N}}\Big(|x|^{-\theta}*|u|^{p_{\gamma}^{*}}\Big)|u|^{p_{\gamma}^{*}}\\
&\geq \|u\|_{H_{0}^{2}(\R^N)}^2-C\|u\|_{H_{0}^{2}(\R^N)}^{2.p_{\theta}}-C_\alpha \|u\|_{H_{0}^{2}(\R^N)}^{2p_{\gamma}^{*}}.
\end{aligned}
$$
Hence, there exists some constant $C_0>0$ such that
\begin{equation}\label{cnot}
\|u\|_{H_{0}^{2}(\R^N)}\geq C_0>0\quad\mbox{for all }u\in {\mathcal N_\alpha}.
\end{equation}
Next, we use the fact that ${\mathcal F}_\alpha \!\mid_{\mathcal N_\alpha}$ is coercive together with\eqref{cnot} in order to obtain
$$
{\mathcal F}_\alpha(u) \geq \Big(\frac{1}{2}-\frac{1}{2p_{\gamma}^{*}}\Big) C_0^2>0.
$$
\end{itemize}
\end{proof}
	
\begin{lemma}\label{nehari2}
If $u$ is a critical point of ${\mathcal F}_\alpha \!\mid_{\mathcal N_\alpha}$, then it will be a free critical point of ${\mathcal F}_\alpha \!\mid_{\mathcal N_\alpha}$.
\end{lemma}
\begin{proof}
Say $ {\mathcal G}(u)=\langle {\mathcal F}_\alpha'(u),u\rangle $ for any $u \in H_{0}^{2}(\R^N)$. Now, for any $u \in {\mathcal N_\alpha} $ we use \eqref{cnot} to get
\begin{equation}\label{cnot1}
\begin{aligned}
\langle {\mathcal G}'(u),u\rangle&=2\|u\|^2-2(p_{\theta})\int_{\R^{N}}\Big(|x|^{-\theta}*|u|^{p_{\theta}}\Big)|u|^{p_{\theta}}-2p_{\gamma}^{*} \alpha\int_{\R^{N}}\Big(|x|^{-\theta}*|u|^{p_{\gamma}^{*}}\Big)|u|^{p_{\gamma}^{*}}\\
&=(2-2p_{\gamma}^{*})\|u\|_{H_{0}^{2}(\R^N)}^2-2(p_{\theta}-p_{\gamma}^{*})\int_{\R^{N}}\Big(|x|^{-\theta}*|u|^{p_{\theta}}\Big)|u|^{p_{\theta}}\\
&\leq -(2p_{\gamma}^{*}-2)\|u\|_{H_{0}^{2}(\R^N)}^2\\
&<-(2p_{\gamma}^{*}-2)C_0.
\end{aligned}
\end{equation}
Now, let us assume that $u\in {\mathcal N_\lambda}$ is a critical point of ${\mathcal F}_\alpha$ in ${\mathcal N_\alpha}$. Using the Lagrange multiplier theorem, we obtain that there exists $\lambda \in \R$ such that ${\mathcal F}_\alpha'(u)=\lambda {\mathcal G}'(u)$. Hence, we get that $\langle {\mathcal F}_\alpha '(u),u\rangle=\lambda \langle {\mathcal G}'(u),u\rangle$. Next, because $\langle {\mathcal G}'(u),u\rangle<0$, which yields $\lambda=0$ and further, we get ${\mathcal F}_\alpha '(u)=0$.
	
\end{proof}

\begin{lemma}\label{nehari3}
Assume that the sequence $(u_n)$ is a $(PS)$ sequence for ${\mathcal F}_\alpha \!\mid_{\mathcal N_\alpha}$. Then $(u_n)$ is a $(PS)$ sequence for ${\mathcal F}_\alpha$.
\end{lemma}
\begin{proof}
Let us assume that $(u_n)\subset {\mathcal N_\alpha}$ be a $(PS)$ sequence for ${\mathcal F}_\alpha \!\mid_{\mathcal N_\alpha}$. Since,
$$
{\mathcal F}_\alpha(u_n)\geq \Big(\frac{1}{2}-\frac{1}{2p_{\gamma}^{*}}\Big)\|u_n\|_{H_{0}^{2}(\R^N)}^2,
$$
which yields $(u_n)$ is bounded in $H_{0}^{2}(\R^N)$. Now, we will prove that ${\mathcal F}'_\alpha(u_n)\to 0$. One could see that,
$$
{\mathcal F}'_\alpha(u_n)- \lambda_n {\mathcal G}'(u_n)= {\mathcal F}'_\alpha \!\mid_{\mathcal N_\alpha}(u_n)= o(1),
$$
for some $\lambda_n \in \R$, which further gives us
$$
\lambda_n \langle {\mathcal G}'(u_n),u_n \rangle= \langle {\mathcal F}_\lambda '(u_n),u_n \rangle + o(1)= o(1).
$$
Next, by the use of \eqref{cnot1}, we deduce that $\lambda_n \to 0$ and this yields ${\mathcal F}_\lambda '(u_n) \to 0$.
\end{proof}

\subsection{A result on compactness}\label{compc}
Let us define the energy functional ${\mathcal R}: H_{0}^{2}(\R^N)\to \R$ by
$$
{\mathcal R}(u)=\frac{1}{2}\|u\|^{2}-\frac{1}{2(p_{\theta})}\intr \Big(|x|^{-\theta}*|u|^{p_{\theta}}\Big)|u|^{p_{\theta}}.
$$
Consider the corresponding Nehari manifold for ${\mathcal R}$ by
$$
{\mathcal N}_{{\mathcal R}}=\{u\in H_{0}^{2}(\R^N) \setminus\{0\}: \langle {\mathcal R}'(u),u\rangle=0\},
$$
and assume that
$$
b_{\mathcal R}=\inf_{u\in {\mathcal N}_{\mathcal R}}{\mathcal R}(u).
$$

And we have,
\begin{equation*}
\begin{aligned}
\langle {\mathcal R}' (u), v \rangle&= \int_{\R^{N}} \Delta u \Delta v  -\int_{\R^{N}}\Big(|x|^{-\theta}*|u|^{p_{\theta}}\Big)|u|^{p_{\theta}-1}v,
\end{aligned}
\end{equation*}
for all $v \in C^{\infty}_{0}(\R^N)$. Also,
\begin{equation*}
\langle {\mathcal R}'(u),u\rangle= \|u\|_{H_{0}^{2}(\R^N)}^2-\int_{\R^{N}}\Big(|x|^{-\theta}*|u|^{p_{\theta}}\Big)|u|^{p_{\theta}}.
\end{equation*}

\begin{lemma}\label{compact}
Suppose that $(u_n)\subset{\mathcal N}_{\mathcal I}$ is a $(PS)$ sequence of ${\mathcal F}_\alpha \!\mid_{{\mathcal N}_{\alpha}}$, that is, $({\mathcal F}_\alpha(u_n))$ is bounded and ${\mathcal F}_\alpha'\!\mid_{{\mathcal N}_{\alpha}}(u_n)\to 0$ strongly in $H_{0}^{-2}(\R^N)$. In such as case, there exists a solution $u\in H_{0}^{2}(\R^N)$ of \eqref{fr} such that, on replacing the sequence $(u_n)$ with the subsequence, we have either of the following alternatives:
	
\smallskip
	
\noindent $(i)$ $u_n\to u$ strongly in $H_{0}^{2}(\R^N)$;
	
or
	
\smallskip
	
\noindent $(ii)$ $u_n\rightharpoonup u$ weakly in $H_{0}^{2}(\R^N)$. Also, there exists a positive integer $l\geq 1$ and $l$ nontrivial  weak solutions to \eqref{squ}, that is, $l$ functions $u_1,u_2,\dots, u_l\in H_{0}^{2}(\R^N)$ and $l$ sequences of points $(q_{n,1})$, $(q_{n,2})$, $\dots$, $(q_{n,l})\subset \R^N$ such that the following conditions hold:
\begin{enumerate}
\item[(a)] $|q_{n,j}|\to \infty$ and $|q_{n,j}-q_{n,i}|\to \infty$  if $i\neq j$, $n\to \infty$;
\item[(b)] $ u_n-\sum_{j=1}^lu_j(\cdot+q_{n,j})\to u$ in $H_{0}^{2}(\R^N)$;
\item[(c)] $ {\mathcal F}_\alpha(u_n)\to {\mathcal F}_{\alpha}(u)+\sum_{j=1}^l {\mathcal R}(u_j)$.
\end{enumerate}
\end{lemma}
\begin{proof}
As we know that $(u_n)\in H_{0}^{2}(\R^N)$ is a bounded sequence, then there exists $u\in H_{0}^{2}(\R^N)$ such that, up to a subsequence, one could have
\begin{equation}\label{firstconv}
\left\{
\begin{aligned}
u_n& \rightharpoonup u \quad\mbox{ weakly in }H_{0}^{2}(\R^N),\\
u_n &\rightharpoonup u\quad\mbox{ weakly in }L^t(\R^N),\; 1\leq t\leq \frac{2N}{N-4},\\
u_n & \to u\quad\mbox{ a.e. in }\R^N.
\end{aligned}
\right.
\end{equation}
We use \eqref{firstconv} together with Lemma \ref{anbl} and get   $${\mathcal F}_\alpha'(u)=0,$$ which further yields that, $u\in H_{0}^{2}(\R^N)$ is a solution of \eqref{fr}. Next, if $u_n\to u$ strongly in $H_{0}^{2}(\R^N)$ then $(i)$ holds.

\medskip
	
Suppose that $(u_n)\in H_{0}^{2}(\R^N)$ does not converge strongly to $u$ and define $e_{n,1}=u_n-u$. In this case $(e_{n,1})$ converges weakly (not strongly) to zero in $H_{0}^{2}(\R^N)$ and
\begin{equation}\label{bl2}
\|u_n\|_{H_{0}^{2}(\R^N)}^2=\|u\|_{H_{0}^{2}(\R^N)}^2+\|e_{n,1}\|_{H_{0}^{2}(\R^N)}^2+o(1).
\end{equation}
Next, we use Lemma \ref{nlocbl} in order to obtain
\begin{equation}\label{bl3}
\int_{\R^N} \Big(|x|^{-\theta}*|u_n|^{p_{\theta}}\Big)|u_n|^{p_{\theta}}=\intr \Big(|x|^{-\theta}*|u|^{p_{\theta}}\Big)|u|^{p_{\theta}}+\intr \Big(|x|^{-\theta}*|e_{n, 1}|^{p_{\theta}}\Big)|e_{n, 1}|^{p_{\theta}}+o(1). 
\end{equation}
By combining \eqref{bl2} and \eqref{bl3} we get
\begin{equation}\label{est6}
{\mathcal F}_\alpha(u_n)= {\mathcal F}_\alpha(u)+{\mathcal R}(e_{n,1})+o(1).
\end{equation}
Using Lemma \ref{anbl}, for any $h\in H_{0}^{2}(\R^N)$, one could have
\begin{equation}\label{est7}
\langle{\mathcal R}'(e_{n,1}), h\rangle=o(1).
\end{equation}
Further, by the use of Lemma \ref{nlocbl} we have
$$
\begin{aligned}
0=\langle {\mathcal F}_\alpha'(u_n), u_n \rangle&=\langle {\mathcal F}_\alpha'(u),u\rangle+\langle {\mathcal R}'(e_{n,1}), e_{n,1} \rangle+o(1)\\
&=\langle{\mathcal R}'(e_{n,1}), e_{n,1}\rangle+o(1).
\end{aligned}
$$
This gives us
\begin{equation}\label{est8}
\langle {\mathcal R}'(e_{n,1}), e_{n,1}\rangle=o(1).
\end{equation}
On the other hand, we also have
$$
\beta:=\limsup_{n\to \infty}\Big(\sup_{q\in \R^N} \int_{B_1(q)}|e_{n,1}|^{\frac{2Np_{\theta}}{2N-\theta}}\Big)> 0.
$$

Hence, we could find $q_{n,1}\in \R^N$ such that
\begin{equation}\label{est9}
\int_{B_1(q_{n,1})}|e_{n,1}|^{\frac{2Np_{\theta}}{2N-\theta}}>\frac{\beta}{2}.
\end{equation}
Therefore, for any sequence $(e_{n,1}(\cdot+q_{n,1}))$, there exists $u_1\in H_{0}^{2}(\R^N)$ such that, up to a subsequence, one could have 
$$
\begin{aligned}
e_{n,1}(\cdot+q_{n,1})&\rightharpoonup u_1\quad\mbox{ weakly in } H_{0}^{2}(\R^N),\\
e_{n,1}(\cdot+q_{n,1})&\to u_1\quad\mbox{ strongly in } L_{loc}^{\frac{2Np_{\theta}}{2N-\theta}}(\R^N),\\
e_{n,1}(\cdot+q_{n,1})&\to u_1\quad\mbox{ a.e. in } \R^N.
\end{aligned}
$$
Now, we pass to the limit in \eqref{est9} and get
$$
\int_{B_1(0)}|u_{1}|^{\frac{2Np_{\theta}}{2N-\theta}}\geq \frac{\beta}{2},
$$
which yields, $u_1\not\equiv 0$. As $(e_{n,1}) \rightharpoonup 0$ weakly in $H_{0}^{2}(\R^N)$, one could obtain that $(q_{n,1})$ is unbounded. Next, passing to a subsequence, we get that $|q_{n,1}|\to \infty$. Further, using \eqref{est8}, we have ${\mathcal R}'(u_1)=0$, which implies that $u_1$ is a nontrivial solution of \eqref{squ}.
Let us define
$$
e_{n,2}(x)=e_{n,1}(x)-u_1(x-q_{n,1}).
$$
Then, in the same manner as before, we have
$$
\|e_{n,1}\|^2=\|u_1\|^2+\|e_{n,2}\|^2+o(1).
$$
By using Lemma \ref{nlocbl} we get
$$
\int_{\R^N} \Big(|x|^{-\theta}*|e_{n,1}|^{p_{\theta}}\Big)|e_{n,1}|^{p_{\theta}}=\int_{\R^N} \Big(|x|^{-\theta}*|u_1|^{p_{\theta}}\Big)|u_1|^{p_{\theta}}+\intr \Big(|x|^{-\theta}*|e_{n,2}|^{p_{\theta}}\Big)|e_{n,2}|^{p_{\theta}}+o(1). 
$$
Thus, we have
$$
{\mathcal R}(e_{n,1})={\mathcal R}(u_1)+{\mathcal R}(e_{n,2})+o(1).
$$
By \eqref{est6}, we get
$$
{\mathcal F}_\alpha (u_n)= {\mathcal F}_\alpha (u)+{\mathcal R}(u_1)+{\mathcal R}(e_{n,2})+o(1).
$$
Again, we use the same approach as above and get
$$
\langle {\mathcal R}'(e_{n,2}),h\rangle =o(1)\quad\mbox{ for any }h\in H_{0}^{2}(\R^N)
$$
and
$$
\langle {\mathcal R}'(e_{n,2}), e_{n,2}\rangle =o(1).
$$
Next, if $(e_{n,2}) \to 0$ strongly, then by taking $l=1$ in the Lemma \ref{compact} we could finish our proof.

\smallskip 

Let us assume that $e_{n,2}\rightharpoonup 0$ weakly (not strongly) in $H_{0}^{2}(\R^N)$ and one could iterate the process and in $l$ number of steps one could find a set of sequences $(q_{n,j})\subset \R^N$, $1\leq j\leq l$ with 
$$
|q_{n,j}|\to \infty\quad\mbox{  and }\quad |q_{n,i}-q_{n,j}|\to \infty\quad\mbox{  as }\; n\to \infty, i\neq j
$$
and $l$ nontrivial solutions  $u_1$, $u_2$, $\dots$, $u_l\in H_{0}^{2}(\R^N)$ of \eqref{squ} such that, by letting
$$
e_{n,j}(x):=e_{n,j-1}(x)-u_{j-1}(x-q_{n,j-1})\,, \quad 2\leq j\leq l,
$$ 
we obtain
$$
e_{n,j}(x+q_{n,j})\rightharpoonup u_j\quad\mbox{weakly in }\; H_{0}^{2}(\R^N)
$$
and
$$
{\mathcal F}_\alpha(u_n)= {\mathcal F}_\alpha(u)+\sum_{j=1}^l {\mathcal R}(u_j)+{\mathcal R}(y_{n,l})+o(1).
$$
Since, ${\mathcal F}_\alpha (u_n)$ is bounded and ${\mathcal R}(u_j)\geq b_{\mathcal R}$, by iterating the process a finite number of times, we get the desired result.
\end{proof}

\begin{lemma}\label{corr1} 
If any sequence is a $(PS)_d$ sequence of ${\mathcal F}_\alpha \! \mid_{{\mathcal N}_\alpha} $, then it is relatively compact for any $d\in (0,b_{\mathcal R})$ . 
\end{lemma}
\begin{proof}
Suppose that $(u_n)$ is a $(PS)_d$ sequence of ${\mathcal F}_\alpha$ in ${{\mathcal N}_\alpha}$. Next, using the Lemma \ref{compact} we obtain ${\mathcal R}(u_j)\geq b_{\mathcal R}$. Then, upto a subsequence $u_n\to u$ strongly in $H_{0}^{2}(\R^N)$, which further yields that $u$ is a solution of \eqref{fr}. 
\end{proof}

\medskip

To complete the proof of Theorem \ref{gstate}, we require the following result.

\begin{lemma}\label{flg}
$$
b_{\alpha}<b_{\mathcal R}.
$$
\end{lemma}
\begin{proof}
Suppose that ground state solution of \eqref{squ} is denoted by $T\in H_{0}^{2}(\R^N)$ and such solution exists (see \cite{AN2016} and references therein). Assume that $sT$ is the projection of $T$ on ${\mathcal N_\alpha}$, that is, $s=s(T)>0$ is the unique real number such that $sT\in {\mathcal N_\alpha}$. As, $T\in {\mathcal N}_{\mathcal R}$ and $sT\in {\mathcal N_\alpha}$, we get
\begin{equation}\label{g1}
||T||^2= \int_{\R^N} \Big(|x|^{-\theta}*|T|^{p_{\theta}}\Big)|T|^{p_{\theta}}
\end{equation}
and
$$
s^2\|T\|^2=s^{2(p_{\theta})}\int_{\R^N} \Big(|x|^{-\theta}*|T|^{p_{\theta}}\Big)|T|^{p_{\theta}}+ \alpha s^{2p_{\gamma}^{*}}\int_{\R^N} \Big(|x|^{-\theta}*|T|^{p_{\gamma}^{*}}\Big)|T|^{p_{\gamma}^{*}}.
$$
One could easily see that $s<1$ from the above two inequalities. Hence, we get
	
\begin{equation*}
\begin{aligned}
b_\alpha \leq {\mathcal F}_\alpha(sT)&=\frac{1}{2}s^{2}\|T\|^{2}-\frac{1}{2(p_{\theta})}s^{2(p_{\theta})}\int_{\R^N} \Big(|x|^{-\theta}*|T|^{p_{\theta}}\Big)|T|^{p_{\theta}}- \frac{\alpha}{2p_{\gamma}^{*}}s^{2p_{\gamma}^{*}} \int_{\R^N} \Big(|x|^{-\theta}*|P|^{p_{\gamma}^{*}}\Big)|T|^{p_{\gamma}^{*}}\\
&= \Big(\frac{s^{2}}{2}-\frac{s^{2(p_{\theta})}}{2(p_{\theta})}\Big)\|T\|^{2}-\frac{1}{2p_{\gamma}^{*}}\Big(s^2||T||^2-s^{2(p_{\theta})}\int_{\R^N} \Big(|x|^{-\theta}*|T|^{p_{\theta}}\Big)|T|^{p_{\theta}}\Big)\\
&= s^{2} \Big(\frac{1}{2}-\frac{1}{2p_{\gamma}^{*}}\Big)\|T\|^{2}+s^{2(p_{\theta})}\Big(\frac{1}{2p_{\gamma}^{*}}-\frac{1}{2(p_{\theta})}\Big)\|T\|^{2}\\
&< \Big(\frac{1}{2}-\frac{1}{2p_{\gamma}^{*}}\Big)\|T\|^{2}+\Big(\frac{1}{2p_{\gamma}^{*}}-\frac{1}{2(p_{\theta})}\Big)\|T\|^{2}\\
&< \Big(\frac{1}{2}-\frac{1}{2(p_{\theta})}\Big)\|T\|^{2} ={\mathcal R}(T)= b_{\mathcal R}.
\end{aligned}
\end{equation*}
Therefore, we get the desired result.
\end{proof}

Now, by Ekeland variational principle, for any $n\geq 1$ there exists $(u_n) \in {\mathcal N}_\alpha$ such that
\begin{equation*}
\begin{aligned}
{\mathcal F}_\alpha(u_n)&\leq b_\alpha+\frac{1}{n} &&\quad\mbox{ for all } n\geq 1,\\
{\mathcal F}_\alpha(u_n)&\leq {\mathcal F}_\alpha(\tilde{u})+\frac{1}{n}\|\tilde{u}-u_n\| &&\quad\mbox{ for all } \tilde{u} \in {\mathcal N}_\alpha \;\;,n\geq 1.
\end{aligned}
\end{equation*}
From here onwards, one could easily obtain that $(u_n) \in {\mathcal N}_\alpha$ is a $(PS)_{b_\alpha}$ sequence for ${\mathcal F}_\alpha$ on ${\mathcal N}_\alpha$. Next, by combining Lemma \ref{flg} and Lemma \ref{corr1} we deduce that, up to a subsequence $u_n \to u$ strongly in $H_{0}^{2}(\R^N)$ which is a ground state solution of the ${\mathcal F}_\alpha$.

\end{document}